\documentclass[12pt]{amsart}
\usepackage{graphicx, overpic}
\usepackage[below]{placeins}
\usepackage[colorlinks=true, linkcolor=blue, citecolor=blue]{hyperref}
\usepackage[]{algorithm2e}
\usepackage[textsize=tiny]{todonotes}
\usepackage[T1]{fontenc}

\usepackage{amsmath,amsthm,amscd,amssymb,eucal}

\usepackage{enumerate, amsfonts, latexsym, color, url}
\usepackage{fancyvrb}
\CustomVerbatimCommand{\codestyle}{Verb}{formatcom=\ttfamily}
\usepackage{epstopdf}

\usepackage{pinlabel}

\makeatletter
\newsavebox{\@brx}
\newcommand{\llangle}[1][]{\savebox{\@brx}{\(\m@th{#1\langle}\)}%
  \mathopen{\copy\@brx\kern-0.5\wd\@brx\usebox{\@brx}}}
\newcommand{\rrangle}[1][]{\savebox{\@brx}{\(\m@th{#1\rangle}\)}%
  \mathclose{\copy\@brx\kern-0.5\wd\@brx\usebox{\@brx}}}
\makeatother

\begin{document}

\newtheorem{theorem}{Theorem}[section]
\newtheorem{lemma}[theorem]{Lemma}
\newtheorem{proposition}[theorem]{Proposition}
\newtheorem{corollary}[theorem]{Corollary}
\newtheorem{conjecture}[theorem]{Conjecture}
\newtheorem{question}[theorem]{Question}
\newtheorem*{qquestion}{Question}
\newtheorem{problem}[theorem]{Problem}
\newtheorem*{claim}{Claim}
\newtheorem*{criterion}{Criterion}

\theoremstyle{definition}
\newtheorem{definition}[theorem]{Definition}
\newtheorem{construction}[theorem]{Construction}
\newtheorem{notation}[theorem]{Notation}
\newtheorem{object}[theorem]{Object}
\newtheorem{operation}[theorem]{Operation}

\theoremstyle{remark}
\newtheorem{remark}[theorem]{Remark}
\newtheorem{example}[theorem]{Example}

\numberwithin{equation}{subsection}

\newcommand\id{\textnormal{id}}

\newcommand\Z{\mathbb Z}
\newcommand\R{\mathbb R}
\newcommand\C{\mathbb C}
\newcommand\CP{\mathbb{CP}}
\newcommand\OO{\mathcal O}
\newcommand\HH{\mathbb H}
\newcommand\CC{\mathbf C}
\newcommand\BB{\mathbf B}
\newcommand\TT{\mathbf T}
\newcommand\PP{\mathcal P}
\newcommand\W{\mathcal W}
\newcommand\RR{\mathcal R}
\newcommand\FF{\mathcal F}
\newcommand\Aut{\textnormal{Aut}}
\newcommand\dist{\textnormal{dist}}
\newcommand\diam{\textnormal{diam}}
\newcommand\eval{\textnormal{eval}}
\title{Quasimorphisms and pseudo-Anosov flows}

\author{Danny Calegari}
\address{University of Chicago \\ Chicago, IL 60637 USA}
\email{dannyc@uchicago.edu}
\author{Jonathan Zung}
\address{Massachusetts Institute of Technology \\ Cambridge, MA 02139 USA}
\email{jzung@mit.edu}
\date{\today}

\begin{abstract}
We describe two connections between the theory of quasimorphisms and
pseudo-Anosov flows without perfect fits on closed hyperbolic 3-manifolds. First
we show that for every such flow $X$, there are quasimorphisms whose coarse
restriction to each flowline of $\tilde{X}$ (the lifted flow in the universal
cover) are uniform quasi-isometries to $\R$ --- such quasimorphisms are said to
be {\em adapted} to $X$; and that the space of quasimorphisms $Q_X$ adapted to
$X$ is an open convex cone in the space of all quasimorphisms on $\pi_1(M)$.
Second, we obtain upper bounds on the exponential growth rate of closed orbits in such
flows, both in the hyperbolic metric and in a word metric; quasimorphisms play a
key role in obtaining the estimates in the second case.
\end{abstract}

\maketitle
\setcounter{tocdepth}{1}
\tableofcontents

\section{Introduction}
Let $X$ be a flow on a closed hyperbolic 3-manifold $M$, and $\tilde{X}$ the lift of $X$ to $\tilde{M}$.
Do all orbits of $X$ progress in the same general direction in $\tilde{M}$? One way to make this precise is
to ask:

\begin{qquestion}
Is there a Lipschitz function $\phi:\HH^3 \to \R$ whose restriction to each flowline
$\ell$ of $\tilde{X}$ is a quasi-isometry $\phi:\ell \to \R$?
\end{qquestion}
A function $\phi$ with this property is said to be {\em adapted} to $X$.
For such a function $\phi$ to exist it is necessary that every flowline of $\tilde{X}$
must be a quasigeodesic; i.e.\/ $X$ is what is known as a {\em quasigeodesic flow}.

Flows are oriented, and therefore so is every flowline of $\tilde{X}$. 
Since these flowlines are quasigeodesics, each is asymptotic 
to a pair of distinct {\em endpoints} in $S^2_\infty$, and since they are oriented,
we may distinguish the {\em positive} and {\em negative} endpoints of each flowline.

If such a function $\phi$ exists, by continuity, and up to replacing $\phi$ by $-\phi$,
each quasi-isometry $\phi:\ell \to \R$ may be assumed to be coarsely 
orientation preserving. In particular, it cannot be
the case that there are flowlines $\ell,\ell'$ of $\tilde{X}$ such that the
positive endpoint of $\ell$ is equal to the negative endpoint of $\ell'$.
A pair of flowlines with this property is said to be {\em anti-aligned}.
Since $M$ is compact, $\tilde{X}$ contains a pair of 
anti-aligned flowlines if and only if it contains a pair of flowlines $\ell,\ell'$
for which the positive endpoint of $\ell$ is equal to the negative endpoint of
$\ell'$ {\em and vice versa}.

One natural class of geodesic flows $X$ without anti-aligned flowlines is the
class of {\em pseudo-Anosov flows without perfect fits}; see e.g.\/
Fenley \cite{Fenley_quasigeodesic} for an introduction to the theory of pseudo-Anosov flows,
and the meaning of perfect fits. 

If $G=\pi_1(M)$, any function $\phi:G \to \R$ Lipschitz in the word metric may be
extended from a $G$-orbit in $\HH^3$ to a Lipschitz function $\phi:\HH^3 \to \R$. 
If $G$ is any group, a {\em quasimorphism} $\phi:G \to \R$ is a function for which 
there is some least $D(\phi)$ so that $|\phi(gh)-\phi(g)-\phi(h)|\le D(\phi)$ for
all $g,h\in G$. Any quasimorphism on $\pi_1(M)$ defines in this way a Lipschitz
function on $\HH^3$. 

Our first main theorem (Theorem~\ref{theorem:uniform_quasimorphism}) says the following.
First of all, if $X$ is a pseudo-Anosov flow
without perfect fits, then there are (unbounded) quasimorphisms on $G$ that are
adapted to $X$. Secondly, such quasimorphisms enjoy an additional geometric
property: they are {\em uniform}, which means that their coarse level sets are
coarsely connected; see \cite{Calegari_Loukidou_zippers}, \S~4 for a discussion of
uniform quasimorphisms. Thirdly, the space $Q_X(G)$ of quasimorphisms on $G$ adapted
to $X$ is an open convex cone in the space $Q(G)$ of all quasimorphisms.

Theorem~\ref{theorem:uniform_quasimorphism} is used by Calegari--Loukidou 
\cite{Calegari_Loukidou_wheels}, \S~6, to show that a pseudo-Anosov flow $X$ can 
be reconstructed from any adapted uniform quasimorphism. Consequently there is a 
canonical correspondence between connected components of the space 
of uniform quasimorphisms on $\pi_1(M)$ and pseudo-Anosov flows without perfect fits
on $M$ up to orbit equivalence.

\medskip

The second part of the paper concerns the exponential growth rate of the 
number of closed orbits of a pseudo-Anosov flow without perfect fits. We 
show that this exponential growth rate is strictly smaller than the exponential 
growth rate of $\pi_1(M)$. 
This is true both with respect to any word metric (Theorem~\ref{theorem:orbit_bound_word})
or the hyperbolic metric (Theorem~\ref{theorem:orbit_bound_hyp}).

The proof of Theorem~\ref{theorem:orbit_bound_word} makes use of the adapted quasimorphisms 
constructed in Theorem~\ref{theorem:uniform_quasimorphism}. The value of any quasimorphism $\phi$ 
on a random element of $\pi_1(M)$ concentrates near $0$, whereas the value of any $\phi$ adapted to $X$
will grow linearly on orbits of $X$. Thus, it is exponentially unlikely for an element of 
$\pi_1(M)$ to be realized by a closed orbit of $X$. To make this precise, we use 
Calegari--Fujiwara's theory of bicombable quasimorphisms \cite{Calegari_Fujiwara}, 
geodesic combings of $\pi_1(M)$, and symbolic codings for $X$ to express $\phi$ in terms of 
the statistics of a random walk on a finite directed graph.

Similar results for arbitrary pseudo-Anosov flows are obtained using rather
different methods by Barthelm\'e--Mann--Paulet--Zalloum \cite{BMPZ}. 

We give two proofs of Theorem~\ref{theorem:orbit_bound_hyp}. The first appeals
to a rather general large deviations estimate that follows from the work of
Cantrell--Reyes--Sert \cite{Cantrell_Reyes_Sert}. The second is more geometric. 
The theory of CaTherine wheels \cite{Calegari_Loukidou_wheels} associates to a 
pseudo-Anosov flow without perfect fits a map 
$f:S^1\to S^2_\infty$ with the property that there is some uniform $K$ so that
the image of each interval $I\subset S^1$ is a $K$-quasidisk. In particular, $2K/(K+1)$
is an upper bound for the Hausdorff dimension of the images $Z^\pm$ in $S^2_\infty$ 
of $\tilde{X}$ under the positive and negative endpoint maps respectively. This Hausdorff
dimension bounds the growth rate of orbits of $X$ by an elementary application of
Ledrappier--Young \cite{Ledrappier_Young_1, Ledrappier_Young_2}. Turning this
around, an estimate on the growth rate of periodic orbits of $X$ gives a {\em lower bound}
on the Hausdorff dimension of the endpoint sets.
Computations (see \S~\ref{section:computations}) suggest that this lower
bound is very close to sharp, at least in some simple examples. 

\medskip

Although we are motivated by applications to the geometry and dynamics of pseudo-Anosov
flows on hyperbolic 3-manifolds, our methods are quite general, and apply without
modification to (word)-hyperbolic groups in the sense of Gromov. Let us briefly explain
this. Let $G$ be a non-elementary hyperbolic group. 
Let $\tilde{\FF}(G)$ denote Mineyev's {\em flow space},
and let $\FF(G)$ denote the quotient by $G$; see
\cite{Mineyev}. The flow space $\tilde{\FF}(G)$ is a metric space homeomorphic to 
$(\partial_\infty G \times \partial_\infty G - \Delta) \times \R$ and quasi-isometric
to $G$, and comes with commuting isometric actions of $G$ on the $\partial_\infty G$
factors, and a {\em flow} on the $\R$ factors. If $G=\pi_1(M)$ for $M$ closed and
negatively curved, $\FF(G)$ is equivariantly homeomorphic to the unit tangent bundle of
$M$ together with its geodesic flow.

There is an involution $\iota:\FF(G) \to \FF(G)$ which anticommutes with $\R$ and
reverses the direction of the flow. Say that $\Lambda \subset \FF(G)$ is 
{\em coherent} if it is closed and flow-invariant, and has the property that
$\Lambda$ is disjoint from $\iota(\Lambda)$. Flowlines of $\tilde{\Lambda}$ in
$\tilde{\FF}(G)$ coarsely track (quasi)geodesics in $G$, and it makes sense to ask
for a quasimorphism $\phi$ on $G$ adapted to $\Lambda$. Our methods show that any
coherent $\Lambda$ admits an open convex cone of adapted quasimorphisms, and the
exponential growth rate of the number of closed orbits of $\Lambda$ is strictly less
than the growth rate of $G$ in the word metric.

\section{Uniform quasimorphisms}\label{section:uniform_quasimorphisms}

\subsection{Quasimorphisms}

A general reference for the theory of quasimorphisms is \cite{Calegari_scl}.

\begin{definition}[Quasimorphism]\label{definition:quasimorphism}
Let $G$ be a group. A {\em quasimorphism} is a function $\phi:G \to \R$ for
which there is some least non-negative real number $D(\phi)$ (called the {\em defect})
so that for all $g,h\in G$ there is an inequality
$$|\phi(gh) - \phi(g) - \phi(h)| \le D(\phi)$$
\end{definition}

Any bounded function is a quasimorphism. Any homomorphism $\phi:G\to \R$ is a 
quasimorphism, and a quasimorphism is a homomorphism if and only if it has defect $0$. 
Nontrivial examples arise from negative curvature (de Rham quasimorphisms;
counting quasimorphisms; see \S~\ref{subsection:counting_quasimorphisms})
and from symplectic/causal geometry (rotation quasimorphisms; Maslov quasimorphisms).

Two quasimorphisms $\phi,\phi'$ are {\em equivalent} if $|\phi - \phi'|$ is a
bounded function. Each equivalence class of quasimorphism contains a unique representative
$\psi$ which is {\em homogeneous}: i.e.\/ that satisfies $\psi(g^n) = n\psi(g)$ for all
$g\in G$ and all $n\in \Z$. Given $\phi$, its homogenization $\psi$ may be defined by
$$\psi(g):=\lim_{n \to \infty} \phi(g^n)/n$$
and satisfies $D(\psi) \le 2D(\phi)$; see \cite{Calegari_scl}.

If $G$ is a group, the set of all equivalence classes of quasimorphisms
on $G$ is denoted $Q(G)$. This contains $H^1(G;\R)$ as a subspace, and the quotient
$Q(G)/H^1(G;\R)$ is a (typically non-separable) Banach space whose norm is the 
defect of the homogeneous representative. Thus when $G$ is finitely generated, 
$Q(G)$ carries a canonical topology as a topological vector space.

\subsection{Counting quasimorphisms}\label{subsection:counting_quasimorphisms}

The theory of counting quasimorphisms on hyperbolic groups was developed by
Fujiwara \cite{Fujiwara}, building on a construction due to Rhemtulla
\cite{Rhemtulla} (rediscovered by Brooks \cite{Brooks}) for free groups,
and we refer to \cite{Fujiwara} for details.

Let $G$ be a finitely generated group with a finite symmetric generating set
$S$, and let $\Gamma$ denote the Cayley graph of $G$ with respect to $S$. We identify
$G$ with the vertices of $\Gamma$, and endow $\Gamma$ with the structure of a
geodesic metric space by setting the length of every edge equal to $1$; denote
this metric by $d_\Gamma(\cdot,\cdot)$. The identity
element $e\in G$ picks out a distinguished basepoint in $\Gamma$. The action
of $G$ on itself from the left induces an isometric simplicial action of $G$ on $\Gamma$.

Let us now assume that $G$ is a hyperbolic group, so that $\Gamma$ is a
$\delta$-hyperbolic space in the sense of Gromov \cite{Gromov} for some finite $\delta$. This means
that for every geodesic triangle in $\Gamma$, every point on an edge of $\Gamma$ is
within distance $\delta$ of some point on one of the other two sides.

If $\alpha$ is a finite (simplicial) path in $\Gamma$, 
denote the (simplicial) length of $\alpha$ by $|\alpha|$. 
If $\alpha,\beta$ are two oriented embedded simplicial paths in $\Gamma$, and
$\alpha$ includes as a subpath of $\beta$ {\em in an orientation preserving way}, we 
write $\alpha \prec \beta$.

If $\sigma$ is an oriented path in $\Gamma$, and $g\in G$,
we let $g\sigma$ denote the oriented path in $\Gamma$ obtained from $\sigma$ 
by left multiplication by $g$, and we let $\sigma^{-1}$ denote the same path
$\sigma$ with the opposite orientation.

\begin{definition}[Small counting quasimorphism]\label{definition:small_counting_quasimorphism}
If $\Delta$ is any set of finite
oriented embedded simplicial paths in $\Gamma$, and $G\Delta$ denotes the set of
paths in $\Gamma$ of the form $g\sigma$ with $g\in G$
and $\sigma \in \Delta$, then for any finite oriented embedded simplicial path
$\gamma$ in $\Gamma$, we define 
$$c_\Delta(\gamma): = \max \text{ number of interior disjoint }\tau \prec \gamma \text{ so that each }
\tau \text{ is in } G\Delta$$

For any $g \in G$ define the {\em small counting function} $c_\Delta(g)$ to be
$$c_\Delta(g): = \max_\gamma c_\Delta(\gamma) + d_\Gamma(e,g) - |\gamma|$$
where the maximum is taken over all oriented embedded paths $\gamma$ from $e$ to $g$.

Finally, define the {\em small counting quasimorphism} $\phi_\Delta(g)$ to be the difference
$$\phi_\Delta(g) = c_\Delta(g) - c_\Delta(g^{-1})$$
\end{definition}
Note that if $\Delta$ is a set of finite oriented embedded simplicial paths in $\Gamma$, we
denote by $\Delta^{-1}$ the same set of paths with the opposite orientation. Evidently,
by the symmetry of the generating set $S$, we have $c_\Delta(g^{-1}) = c_{\Delta^{-1}}(g)$.

The idea of the definition of $c_\Delta(g)$ is that associated to an element $g\in G$
we try to find an oriented path $\gamma$ from $e$ to $g$ that contains as many translates
of elements of $\Delta$ as possible, while at the same time is as short as possible.
A path $\gamma$ from $e$ to $g$ maximizing $c_\Delta(\gamma) + d_\Gamma(e,g) - |\gamma|$
is called a {\em realizing path} for $g$.

\begin{lemma}[realizing paths quasigeodesic \cite{Fujiwara}, Lemma~3.3]\label{lemma:realizing_path_quasigeodesic}
Suppose that every element $\sigma$ of $\Delta$ has $|\sigma| \ge 2$. Then for any
$g\in G$ any realizing path $\gamma$ is a $(2,4)$-quasigeodesic.

Consequently there is a constant $C(\delta)$ such that for any $g\in G$, if the
$C(\delta)$-neighborhood of any oriented geodesic from $e$ to $g$ does not contain a
{\em coherently oriented} translate $g\sigma$ of some $\sigma \in \Delta$, then
$c_\Delta(g)=0$.
\end{lemma}

We emphasize that the different translates of $\Delta$ elements in $\gamma$ are allowed to be
translates of different elements of $\Delta$, but they are required to be disjoint
(the disjointness accounts for the adjective `small'). Allowing overlapping translates
defines the {\em big counting quasimorphism}. Both are, indeed, quasimorphisms, justifying
the name. One advantage of working with small counting quasimorphisms is that their defects
can be bounded independently of $\Delta$: 
\begin{proposition}[Defect bound \cite{Fujiwara}, Proposition~3.10]\label{proposition:defect_bound}
Suppose that every element $\sigma$ of $\Delta$ has $|\sigma|\ge 2$. Then
the function $\phi$ is a quasimorphism on $G$ with defect $D(\phi) \le C(\delta)$.
\end{proposition}

The proofs of Lemma~\ref{lemma:realizing_path_quasigeodesic} and
Proposition~\ref{proposition:defect_bound} in \cite{Fujiwara} treat only the case 
that $\Delta$ contains a single element, but nowhere use this hypothesis. 
A crude estimate gives a bound of the order of magnitude $C(\delta) \le 20\delta$ 
but there is probably no reason to try to make it sharp.

\subsection{Counting quasimorphisms from flows}

Let $M$ be a closed hyperbolic 3-manifold with fundamental group $G$. If we fix a 
finite symmetric generating set for $G$ then we may form the Cayley graph $\Gamma$ as
in \S~\ref{subsection:counting_quasimorphisms} and observe that this is $\delta$-hyperbolic 
for some $\delta$.

Pick a basepoint $p \in M$ and a lift $\tilde{p} \in \tilde{M} = \HH^3$ and use
the $G$ action to build an equivariant quasi-isometry $\Gamma \to \HH^3$. For concreteness,
we can start with a 1-vertex geodesic triangulation of $M$, and take $\Gamma$ to
be the $1$-skeleton of the lift of this triangulation to $\HH^3$.
Fix a compact fundamental domain $A \subset \HH^3$.

Let $X$ be a pseudo-Anosov flow without perfect fits on $M$, and let $\tilde{X}$ be
the lifted flow in the universal cover. It is convenient to assume that $X$ is smooth;
if not, we can modify it slightly to obtain a nearby flow which is smooth and quasigeodesic,
and has no perfect fits; these are the only properties of $X$ we use in the sequel.

\begin{lemma}[Uniform quasigeodesics \cite{Fenley_quasigeodesic}]\label{lemma:flowlines_quasigeodesic}
The flowlines of $\tilde{X}$ are uniform $K$-quasigeodesics and therefore there is a constant $C(K)$
so that every flowline $\ell$ of $\tilde{X}$ is contained in the $C(K)$-neighborhood of
some (any) geodesic $\ell_\Gamma$ in $\Gamma$.
\end{lemma}
We call $\ell_\Gamma$ as in the statement of Lemma~\ref{lemma:flowlines_quasigeodesic} a
{\em straightening} of $\ell$.

\begin{lemma}[No long anti-aligned segments]\label{lemma:no_anti_align}
For every $C$ there is a constant $T(C)$ so that if $\ell$ is a positively oriented
flowline of $\tilde{X}$ with oriented straightening $\ell_\Gamma$, and if
$\ell'$ is a negatively oriented flowline of $\tilde{X}$ with
oriented straightening $\ell'_\Gamma$, then if there is an oriented segment
$\sigma' \prec \ell'_\Gamma$ and a coherently oriented translate $g\sigma'$
contained in the $C$-neighborhood of $\ell_\Gamma$
then $|\sigma'| \le T(C)$.
\end{lemma}
\begin{proof}
Suppose not. Then we may extract a pair of flowlines $\ell,\ell'$ so that the
positive endpoint of $\ell$ is equal to the negative endpoint of $\ell'$. But this
implies that $X$ has a perfect fit.
\end{proof}

Fix some $T \gg T(C)$ where $C/100 > \max{\diam(A),C(K),C(\delta)}$ where
$A$ is the chosen fundamental domain for $M$, where $C(K)$ is the constant from
Lemma~\ref{lemma:flowlines_quasigeodesic}, and where $C(\delta)$ is the
constant from Lemma~\ref{lemma:realizing_path_quasigeodesic} for $\delta$ the
constant of $\delta$-hyperbolicity associated to $\Gamma$.

Define $\Delta$ to be the set of all (oriented) segments $\sigma$ of length $T$
contained in all oriented straightenings $\ell_\Gamma$ of all flowlines of $\tilde{X}$.
Note that this is only finitely many $G$ orbits of finite oriented segments.
Define $\phi_\Delta$ to be the small counting quasimorphism associated to $\Delta$
in Definition~\ref{definition:small_counting_quasimorphism}.

\begin{proposition}[Linear on flowlines]\label{proposition:linear_on_flowlines}
There is a constant $C>0$ so that 
for any oriented flowline $\ell$ of $\tilde{X}$ and any oriented straightening
$\ell_\Gamma$ whose vertices are enumerated as $g_i$ for $i\in \Z$, there is an
inequality
$$T^{-1} (j-i) - C \le \phi_\Delta(g_j) - \phi_\Delta(g_i) \le T^{-1} (j-i) + C$$
\end{proposition}
\begin{proof}
For $i<j$ let $g_i^{-1}\sigma(i,j)$ be the oriented segment of $\ell_\Gamma$ from $e$ to 
$g_i^{-1}g_j$. By definition, this segment contains $\lfloor |j-i|/T\rfloor$ interior
disjoint oriented translates of elements of $\Delta$, and is a geodesic, 
and therefore it is a realizing path for $g_i^{-1}g_j$, so that we actually have
equality $c_\Delta(g_i^{-1}g_j) = \lfloor |j-i|/T\rfloor$. On the other hand, by our
choice of $T$ and by Lemma~\ref{lemma:realizing_path_quasigeodesic}
the $C(\delta)$ neighborhood of $\sigma(i,j)$ cannot contain a coherently oriented
translate of $\sigma^{-1}$ for any $\sigma \in \Delta$, and therefore
$c_\Delta^{-1}(g_i^{-1}g_j) = 0$ so that $\phi_\Delta(g_i^{-1}g_j) = \lfloor |j-i|/T\rfloor$
and consequently by the quasimorphism property,
$$-C(\delta) \le \phi_\Delta(g_j) - \phi_\Delta(g_i) - \lfloor |j-i|/T\rfloor \le C(\delta)$$
for $C(\delta)$ as in Proposition~\ref{proposition:defect_bound}.
\end{proof}

\begin{definition}[Adapted quasimorphism]
Let $X$ be a pseudo-Anosov flow without perfect fits on $M$. A quasimorphism $\phi:G \to \R$
is {\em adapted to $X$} if its restriction to the straightening of each oriented
flowline of $\tilde{X}$ is a uniform (coarsely increasing) quasi-isometry to $\R$.
\end{definition}

Thus Proposition~\ref{proposition:linear_on_flowlines} says that $\phi_\Delta$ is
adapted to $X$.

\subsection{Mollifying $\phi$}

\begin{proposition}[Mollifying $\phi$]\label{proposition:mollify}
Let $X$ be a pseudo-Anosov flow without perfect fits on $M$.
Let $\phi:G \to \R$ be any quasimorphism which is adapted to $X$. Then there is a function
$\psi:\HH^3 \to \R$ which satisfies the following properties:
\begin{enumerate}
\item{$|\phi-\psi|$ is bounded on the vertices of $\Gamma$;}
\item{$\psi$ is Lipschitz;}
\item{the restriction of $\psi$ to each flowline of $\tilde{X}$ is monotone nondecreasing.}
\end{enumerate}

In particular, $\phi$ is uniform; i.e.\/ its coarse level sets are coarsely connected.
\end{proposition}
\begin{proof}
Recall that we have chosen $\Gamma$ to be the 1-skeleton of the lift to $\HH^3$ of a
geodesic triangulation of $M$. We may therefore extend $\phi$ to a continuous function
on $\HH^3$ (which, by abuse of notation, we continue to call $\phi$)
by making it linear on each simplex. This function is evidently Lipschitz
because $\phi$ is a quasimorphism, and $M$ is compact.

Define a new function $\psi:\HH^3 \to \R$ as follows. For each oriented flowline
$\ell$ of $\tilde{X}$ and each $x \in \ell$ define
$$\psi(x) = \max \phi(y) \text{ for } y\in \ell \text{ before } x$$
The flowline $\ell$ stays within $C(K)$ of its straightening $\ell_\Gamma$
(Lemma~\ref{lemma:flowlines_quasigeodesic}) and the restriction of $\phi$ to 
$\ell_\Gamma$ is approximately linear increasing 
(Proposition~\ref{proposition:linear_on_flowlines}), and therefore
$|\psi-\phi|$ is bounded and $\psi$ is Lipschitz.

It suffices to show $\phi$ is uniform.
Let $g,h \in G$ satisfy $\phi(g)=\phi(h)=0$ (for simplicity). 
Since $\psi$ is Lipschitz and $|\phi-\psi|$ is bounded where defined, we
may join $g$ and $h$ by paths $\alpha,\beta$ of bounded length to some points
$p,q$ with $\psi(p)=\psi(q)=0$. The level set $\psi=0$ is connected by construction,
so there is a path $\gamma$ in this level set from $p$ to $q$. Then
$\alpha \cup \gamma \cup \beta$ is a path from $g$ to $h$ which may be approximated
(since $\psi$ is Lipschitz) by a sequence $g=g_0,g_1,\cdots,g_n=h$ with
$d_\Gamma(g_i,g_{i+1})<C_1$ and $|\phi(g_i)|\le C_2$ for constants
$C_1,C_2$ independent of $g$ and $h$.
\end{proof}

\subsection{Proof of the main theorem}

\begin{theorem}[Uniform quasimorphism]\label{theorem:uniform_quasimorphism}
Let $M$ be a closed hyperbolic 3-manifold with fundamental group $G$, 
and let $X$ be a pseudo-Anosov flow without perfect fits, and let
$\tilde{X}$ be the lifted flow to the universal cover. Then there is a 
(necessarily uniform) quasimorphism $\phi:G \to \R$ adapted to $X$.

Moreover, if $Q_X(G)$ denotes the set of uniform quasimorphisms on $G$ adapted to $X$,
then $Q_X(G)$ is an open convex cone in $Q(G)$, the space of quasimorphisms on 
$G$ modulo bounded functions.
\end{theorem}
\begin{proof}
We have already proved the first part, namely we have constructed a 
quasimorphism $\phi_\Delta$ adapted to $X$ by Proposition~\ref{proposition:linear_on_flowlines}. 
And deduced that it is uniform by Proposition~\ref{proposition:mollify}.

Evidently the space of quasimorphisms adapted to $X$ forms a convex cone, since
a convex sum of their mollifications is evidently uniform and adapted to $X$.
Finally, for any uniform quasimorphism $\phi$ adapted to $X$ we may obtain $\psi$ 
as in Proposition~\ref{proposition:mollify} by mollification. Then if
$\xi$ is an arbitrary quasimorphism, for sufficiently small $\epsilon$ the sum
$\phi + \epsilon \xi$ is still adapted to $X$ and is therefore uniform.
\end{proof}

\subsection{Adapted quasimorphisms for geodesic laminations}

It seems worthwhile to identify exactly what properties of $X$ are really used to
obtain quasimorphisms adapted to $X$ (leaving aside for now the stronger conclusion that
such quasimorphisms are uniform).

First of all, we never use either that flowlines of $X$ are embedded, or that their union is all
of $M$. Flowlines of $X$ are quasigeodesic, and by compactness of $M$ uniformly quasigeodesic,
so we may as well work directly in the unit tangent bundle $UTM$ and replace flowlines
of $X$ by the geodesics they uniformly fellow travel, i.e. we can work with any
$\Lambda \subset UTM$, a closed, geodesic flow-invariant subset; such objects are known
as oriented geodesic laminations. 
Finally, we need to know that there is no pair of geodesics $\ell,\ell'$ of $\tilde{\Lambda}$
so that $\ell$ is asymptotic in forward time to the same point in $S^2$ that
$\ell'$ is asymptotic to in negative time; equivalently, by compactness of 
$M$ and $\Lambda$, there are no pair of geodesics $\ell$ and $\ell'$ of $\tilde{\Lambda}$
so that $\ell'$ is equal to $\ell$ with the opposite orientation.

All of these properties and our arguments generalize immediately to geodesic laminations
in arbitrary word-hyperbolic groups. Let $G$ be an arbitrary word-hyperbolic group.
Let $\tilde{\FF}(G)$ denote Mineyev's {\em flow space}, and let $\FF(G)$ denote the 
quotient by $G$; see \cite{Mineyev}, \S~13 especially Theorem~60. 
The flow space $\tilde{\FF}(G)$ is a metric space homeomorphic to 
$(\partial_\infty G \times \partial_\infty G - \Delta) \times \R$ (Theorem~60, bullet (a))
and quasi-isometric to $G$ (Theorem~60, bullet (c)), and comes with a
commuting isometric diagonal action of $G$ on the $\partial_\infty G$
factors, and a {\em flow} on the $\R$ factor (Theorem~60, bullet (e)). 
If $G=\pi_1(M)$ for $M$ closed and
negatively curved, $\FF(G)$ is equivariantly homeomorphic to the unit tangent bundle of
$M$ together with its geodesic flow. Let $\iota:\FF(G) \to \FF(G)$ be the isometric
involution that reverses the direction of the flow and anti-commutes with it 
(Theorem~60, bullet (f)).

\begin{definition}[Coherent]\label{definition:coherent}
A closed, flow-invariant subset $\Lambda \subset \FF(G)$ is an {\em oriented geodesic 
lamination} of $G$. 

An oriented geodesic lamination is {\em coherent} if $\Lambda$ is disjoint 
from $\iota \Lambda$.
\end{definition}

Any flow-invariant subset $\Lambda$ of $\FF(G)$ is covered by a flow-invariant
subset $\tilde{\Lambda}$ of $\tilde{\FF}(G)$. Flowlines of $\tilde{\Lambda}$ coarsely
track oriented quasigeodesics in $G$. A quasimorphism $\phi:G \to \R$ is
{\em adapted} to $\Lambda$ if its restriction to each flowline $\ell$ of $\tilde{\Lambda}$
is a coarse orientation-preserving quasi-isometry from $\ell$ to $\R$.

\begin{theorem}[Adapted quasimorphisms for laminations]\label{theorem:adapted_lamination}
Let $G$ be a hyperbolic group and let $\Lambda \subset \FF(G)$ be an oriented geodesic lamination. 

Then there is a quasimorphism $\phi:G \to \R$
adapted to $\Lambda$ (i.e.\/ such that for each oriented leaf $\ell$ of $\tilde{\Lambda}$
the coarse map $\phi:\ell \to \R$ is a (uniform in $\ell$) coarsely increasing
quasi-isometry to $\R$) if and only if $\Lambda$ is coherent.

Moreover, if $Q_\Lambda(G)$ denotes the space of (equivalence classes of) 
quasimorphisms on $G$ adapted to $\Lambda$,
then $Q_\Lambda(G)$ is an open convex cone in $Q(G)$.
\end{theorem}
\begin{proof}
The proof is almost identical to that of Theorem~\ref{theorem:uniform_quasimorphism}.
The condition that $\Lambda$ is disjoint from $\iota(\Lambda)$ is evidently necessary;
conversely, by compactness, if $\Lambda$ is disjoint from $\iota(\Lambda)$ then if we
fix a fundamental domain $A$ in $\tilde{\FF}(G)$, there is a $C$ so that two leaves
of $\tilde{\Lambda}$ can only stay within distance $\diam(A)$ of each other along
anti-aligned segments for distance $C$. Now choose $T\gg C,\diam(A),C(\delta)$ and build a (small)
counting quasimorphism $\phi$ that counts non-overlapping copies of all geodesic segments
in the Cayley graph of $G$ that stay within distance $C(\delta)$ of an oriented flowline 
of $\tilde{\Lambda}$ for distance $T$. The quasimorphism $\phi$ is evidently adapted
to $\Lambda$. Furthermore, the set $Q_\Lambda(G)$ of quasimorphisms adapted to $\Lambda$
is an open convex cone in $Q(G)$ for the same reasons as in the proof of 
Theorem~\ref{theorem:uniform_quasimorphism}.
\end{proof}

\section{Orbit counts}

\subsection{Bounding orbit counts}

Let $M$ be a closed hyperbolic 3-manifold and let $X$ be a pseudo-Anosov flow on $M$ 
without perfect fits. In this section we prove two theorems that each give a gap 
for the exponential growth rate of the number of closed orbits in $X$ versus the
exponential growth rate of $\tilde{M}$. The first (Theorem~\ref{theorem:orbit_bound_word}) 
fixes a generating set for $\pi_1(M)$ and measures closed orbits by word length;
the second (Theorem~\ref{theorem:orbit_bound_hyp}) measures closed orbits by hyperbolic length.

The proof of the first theorem uses the uniform quasimorphisms constructed in 
\S~\ref{section:uniform_quasimorphisms}
together with the theory of bicombable functions developed in  \cite{Calegari_Fujiwara}.

\begin{theorem}\label{theorem:orbit_bound_word}
Let $M$ be a closed hyperbolic 3-manifold with fundamental group $G$, and 
let $X$ be a pseudo-Anosov flow 
without perfect fits. Fix a symmetric generating set $S$ for $G$, let $G_T$ be the
radius $T$ ball in the word metric with respect to this generating set, and 
let $\log \lambda_G: = \lim_{T\to \infty} {\log |G_T|}/{T}$.
For each real $T$ let $N(X,T)$ be the number of elements of $G_T$
representing closed orbits of $X$. 
Then there is  
$D = \log \lambda_X <\log \lambda_G$ so that
$$\lim_{T \to \infty} \frac {\log N(X,T)} {T} = D$$
Furthermore, $\lambda_X$ and $\lambda_G$ are algebraic numbers.
\end{theorem}
\begin{proof}
We may obtain the weaker estimate 
$\lim_{T \to \infty} {N(X,T)}/{|G_T|} = 0$
by using the results of \cite{Calegari_Fujiwara} as a black box. We first explain this
argument.

By Theorem \ref{theorem:uniform_quasimorphism}, there is an adapted 
quasimorphism $\phi:G\to \R$; in fact, the quasimorphism constructed is
a counting quasimorphism, and hence is bicombable \cite[Theorem 3.18]{Calegari_Fujiwara}. 
Calegari and Fujiwara prove a central limit theorem for the value of such quasimorphisms on 
random elements of $G_T$. In particular, \cite[Corollary 4.26]{Calegari_Fujiwara} 
states that there are constants $E,K$ and a subset $G'_T \subset G_T$ with 
$|G'_T|/|G_T|=1-o(1)$ so that for all $g\in G'_T$, there is an inequality 
$$|\phi(g)-nE| \leq K\sqrt n.$$ 
For any quasimorphism 
$|\phi(g) + \phi(g^{-1})|$ is bounded independent of $g$, so we must have $E=0$. 
On the other hand, the value of $\phi$ grows 
linearly on flowlines. Therefore, for sufficiently large $T$, most elements of $G_T$
do not represent closed orbits of $X$.

\medskip

To get the more refined estimate we need to understand a little more
about bicombable quasimorphisms and bicombable functions in general. If $G$ is any
group with generating set $S$, a {\em prefix-closed geodesic combing} is a
prefix-closed regular language $L\subset S^*$ whose evaluation map is a bijection $L \to G$
taking elements of $L$ of length $n$ to elements of $G$ of word length $n$ (in the
generating set $S$); see e.g.\/ \cite{Calegari_Fujiwara}, \S~3 for definitions.

Since $L$ is a regular language, one can construct a finite directed graph (hereafter digraph) 
$\Gamma$ with a unique initial vertex (hereafter pointed digraph) 
so that elements in $L$ of length $n$ are in bijection
with directed paths in $\Gamma$ (starting at the initial vertex) of length $n$. The
directed edges of $\Gamma$ are labeled by elements of $S$, and the correspondence
takes a path $\gamma$ to the word in $S^*$ obtained by reading the labels on the edges
that $\gamma$ traverses, in order. 

It is important to say that many different labeled digraphs $\Gamma$ parameterize (in this
way) the same regular language. Let $\Gamma^*$ denote the set of finite directed paths
in $\Gamma$ starting at the initial vertex and $\eval:\Gamma^* \to G$ the evaluation map. 
 
The definition of a ($\Z$-valued) bicombable function $\phi$ 
\cite[Definition~3.4]{Calegari_Fujiwara} is
that there is a digraph $\Gamma$ parameterizing $L$ and a function $d\phi$ from vertices
of $\Gamma$ to $\Z$ so that for any $\gamma \in \Gamma^*$ there is equality
$$\phi(\eval(\gamma)) = \sum d\phi(\gamma(i))$$
where $\gamma(i)$ denotes the successive vertices of $\gamma$.

Any finite directed graph $\Gamma$ determines an associated {\em acyclic} directed graph
$C(\Gamma)$ whose vertices are the {\em recurrent components} of $\Gamma$ --- i.e.\/
the maximal (directed) subgraphs in which there is a directed edge between every two
distinct vertices. Each component
$\Delta$ of $C(\Gamma)$ (i.e.\/ each vertex of $C(\Gamma)$) has an associated 
adjacency matrix $A_\Delta$ whose Perron-Frobenius eigenvalue $\lambda_\Delta$ has the
property that the number of directed paths in $\Delta$ of length $n$ is 
$\Theta(\lambda_\Delta^n)$. A component is {\em maximal} if $\lambda_\Delta = \lambda_G$
as above. These are standard facts in the theory of finite Markov chains; see
e.g.\/ \cite{Kemeny_Snell}.

For each maximal component $\Delta$ let $\xi_\Delta$ be the Perron-Frobenius eigenvector
of $A_\Delta$, normalized to a probability measure supported on the vertices of $\Delta$,
and let $E_\Delta(\phi)$ be the expectation of $d\phi$ with respect to this measure.
\cite[Lemma~4.24]{Calegari_Fujiwara} says that for {\em every} maximal component,
$E_\Delta(\phi) = E$ as above; in particular, if $\phi$ is a bicombable
quasimorphism, $E_\Delta(\phi) = 0$ for every $\Delta$.

The choice of $\Gamma$ is not unique; for any pointed digraph $\Gamma$ and any integer
$T$ we may form a new pointed digraph $\Gamma_T$ (the {\em $T$th iterated edge digraph} 
of $\Gamma$) whose vertices are the set of
directed paths in $\Gamma$ of length $T$ (starting at any vertex), together with
the set of directed paths in $\Gamma$ of length $\le T$ starting at the initial vertex. 
The initial vertex of $\Gamma_T$ corresponds
to the initial vertex of $\Gamma$ (which can be thought of as the unique directed path
of length $0$). If $\gamma \in \Gamma^*$ is a path of length $k \le T$ and $\gamma'$ is
its prefix of length $k-1$ then there is a directed edge in $\Gamma_T$ from
$\gamma'$ to $\gamma$. Otherwise if $\gamma',\gamma$ are any two directed paths of
length $T$ in $\Gamma$ (starting at any vertex), and the suffix of $\gamma'$ of
length $T-1$ equals the prefix of $\gamma$ of length $T-1$, there is an edge in
$\Gamma_T$ from $\gamma'$ to $\gamma$. There is an obvious bijection between
$\Gamma^*$ and $\Gamma_T^*$. If $\Gamma$ is a digraph parameterizing $\phi$, then
so is $\Gamma_T$ in the obvious way. Notice that $\Gamma$ and $\Gamma_T$ both determine
the same geodesic combing $L$ of $G$.

Let's choose $T$ large; we will say how large in a moment.
Fix a fundamental domain $Q\subset \tilde{M}$. Since flowlines of $\tilde{X}$ are uniform
quasigeodesics, there is a uniform constant $C$ so that
for each flowline $\ell$ of $X$ passing through $Q$ there is
a path in $\Gamma^*$ (equivalently, in $\Gamma_T^*$)\
that stays within distance $C$ of $\ell$. Thus we may push down geodesics in $G$ approximating
longer and longer segments in $\ell$ to longer and longer paths in $\Gamma_T$. The
union of the images of all these paths defines a pointed subdigraph $\Gamma_{T,X}$ of
$\Gamma_T$.

On the other hand we claim that for any $C$ there is a $T$ so that if
$\gamma$ is any geodesic path in $G$ (finite or infinite) for which every subpath of 
length $T$ is within distance $C$ of some flowline of $\tilde{X}$, then $\gamma$ itself
is within distance $C$ of some flowline of $\tilde{X}$; in fact, this is just the shadowing lemma
see e.g.\/ \cite[Theorem 18.1.2]{Katok}; a discussion of shadowing (equivalently, the existence
of Markov partitions) for pseudo-Anosov flows specifically may be found in \cite{Jezequel_Zung}.
It follows that if we choose such an $T$, paths in $\Gamma_{T,X}^*$ parameterize 
all and only geodesic paths in $G$ in the given combing that stay within distance $C$
of flowlines of $\tilde{X}$. Morally speaking, this is just a version of the well-known
symbolic coding of a hyperbolic dynamical system coming from a Markov partition.
In particular, we can estimate the number of closed
orbits of $X$ of word length $N$ by counting paths in $\Gamma_{T,X}$ of length $N$.
If maximal components of $C(\Gamma_{T,X})$ have Perron--Frobenius
eigenvalue $\lambda_X$ then the number of directed paths in $\Gamma_{T,X}^*$ of
length $N$ is $\Theta(N^{k-1} \lambda_X^N)$ where $k$ is the length of the longest
sequence of maximal components of $C(\Gamma_{T,X})$ in series; again this is
standard, see \cite{Kemeny_Snell}.

Now, by the very definition of $\phi$, for every maximal component 
$\Delta_X$ of $\Gamma_{T,X}$ the average $E_{\Delta_X}(\phi)$ is strictly positive.
It follows either that $\Delta_X$ is a non-maximal component of $C(\Gamma_T)$, or
that it is strictly contained in some component of $C(\Gamma_T)$. In either case
it follows that there is a strict inequality $\lambda_X < \lambda_G$.

Since $\lambda_X$ and $\lambda_G$ as above are eigenvalues of non-negative integer matrices,
they are algebraic numbers.
\end{proof}

Just as Theorem~\ref{theorem:uniform_quasimorphism} generalized to 
Theorem~\ref{theorem:adapted_lamination} with $\Lambda$ in place of $X$,
Theorem~\ref{theorem:orbit_bound_word} generalizes to the following:

\begin{theorem}[Lamination orbit bound]\label{theorem:orbit_bound_lamination}
Let $G$ be a hyperbolic group and let $\Lambda\subset \FF(G)$ be an oriented geodesic
lamination, thought of as a closed, flow-invariant subset of Mineyev's flowspace. 
Suppose that $\Lambda$ is coherent, i.e.\/ that $\Lambda$ is
disjoint from $\iota \Lambda$ where $\iota$ is the involution of $\FF(G)$ that
reverses the orientation of the flowlines.

Fix a symmetric generating set $S$ for $G$, let $G_T$ be the
radius $T$ ball in the word metric with respect to this generating set, and 
let $\log \lambda_G: = \lim_{T\to \infty} {\log |G_T|}/{T}$.
For each real $T$ let $N(\Lambda,T)$ be the number of elements of $G_T$
representing closed orbits of $\Lambda$. 
Then there is $D = \log \lambda_\Lambda <\log \lambda_G$ so that
$$\lim_{T \to \infty} \frac {\log N(\Lambda,T)} {T} = D$$
\end{theorem}
\begin{proof}
The existence of a bicombable quasimorphism adapted to $\Lambda$ is
Theorem~\ref{theorem:adapted_lamination}. We may now repeat the proof
of Theorem~\ref{theorem:orbit_bound_word} with $\Lambda$ in place of $X$. 

The only subtlety is that the shadowing lemma does not hold for a general
closed flow-invariant subset of $UTM$; this means that the growth rate of
paths in $\Gamma_{T,\Lambda}$ might be bigger than the growth rate of the
number of closed orbits of $\Lambda$. Thus the inequality still holds
in this case, though we do not obtain a precise estimate of $\lambda_\Lambda$
from $\Gamma_{T,\Lambda}$.
\end{proof}

\begin{remark}
The structure of Theorem~\ref{theorem:orbit_bound_word} and its proof are very similar to 
Landry--Minsky--Taylor \cite{Landry_Minsky_Taylor} Theorems~7.1 and 7.2;
in either case the exponential growth rate of the number of closed orbits
of a pseudo-Anosov flow without perfect fits (in a certain combinatorial
metric) is obtained as the 
Perron-Frobenius eigenvalue of an adjacency matrix associated to a finite
digraph. In the case of \cite{Landry_Minsky_Taylor} this is the flow graph
coming from a veering triangulation. In \cite{Landry_Minsky_Taylor}
this growth rate can be refined to count flowlines in a particular 
(projective) homology class; we remark that such conditional
counts are easy to obtain by our methods too,
since a projective integral homology class is just the kernel of a 
homomorphism $\alpha:\pi_1(M) \to \Z^{b_1(M)-1}$ whose coefficients are
(obviously) bicombable, and therefore we may estimate growth of
flowlines in a projective homology class in
the fiber product of digraphs for $\phi$ and for the coefficients of $\alpha$.
Of course carrying this out in practice is a vastly different enterprise:
one of the main advantages of the method of \cite{Landry_Minsky_Taylor} is
that their invariant is effectively computable in practice.
\end{remark}

\subsection{Hyperbolic length}

In this subsection we give an analog of Theorem~\ref{theorem:orbit_bound_word} with
hyperbolic length in place of word length.

\begin{theorem}[Orbit bound]\label{theorem:orbit_bound_hyp}
Let $M$ be a closed hyperbolic 3-manifold, and let $X$ be a pseudo-Anosov flow on $M$
without perfect fits. For each real $T$ let $N(X,T)$ be the number of closed orbits
of $X$ whose geodesic representative in $M$ has length $\le T$. Then there is a
$D < 2$ so that 
$$\lim_{T \to \infty} \frac {\log N(X,T)} {T} = D$$
\end{theorem}

This theorem follows rather immediately as a corollary of a large deviation result
due to Cantrell--Reyes--Sert \cite{Cantrell_Reyes_Sert},
valid for {\em any} quasimorphism on any hyperbolic group and with respect to 
{\em any} metric on the group which is invariant, hyperbolic, and quasi-isometric to
a word metric (another proof should follow from Ruelle's thermodynamic formalism 
\cite{Ruelle} applied to a cocycle measuring the difference between word and
hyperbolic length, as in the work of Pollicott--Sharp \cite{Pollicott_Sharp}, with
spectral gap theorems of Parry--Pollicott \cite{Parry_Pollicott} giving the desired
exponential separation). 

We give a second proof that rests on the theory of
CaTherine wheels and which gives an explicit bound on $D$ in terms of the Hausdorff
dimension of a $G$-zipper; see \cite{Calegari_Loukidou_wheels}.

\begin{proof}
Let's denote $G=\pi_1(M)$, let's fix a basepoint $p\in \HH^3$, and for any positive
constant $T$, let's let $G_T$ denote the set of $g\in G$ with $d_{\HH^3}(p,gp)\le T$.
Fix a quasimorphism $\phi$ adapted to $X$ and for any positive $C$ let's let $G_T(\phi,C)$
denote the set of elements $g\in G_T$ with $\phi(g) > C\cdot T$. We want to know that
the exponential growth rate of $G_T(\phi,C)$ is strictly less than the exponential
growth rate of $G_T$ (which is $2$). This is a count with respect to displacement rather
than length in a conjugacy class, but it gives an upper bound up to a change of
constants. 

The main result we want to use is \cite{Cantrell_Reyes_Sert}, Theorem~6.1 which
gives lower and upper bounds for $\lim 1/T \log |G_T(\phi,C)|/|G_T|$ in terms of
the values of a function $-I$ which is the Legendre transform of $\theta$, the
spectral parameterization of the Manhattan curve. Proposition~5.10 says that 
the function $\theta$ is
strictly convex, and the gradient function $\nabla \theta$ is injective; since
the ratio goes to $1$ as $C \to 0$ it follows that for $C$ strictly positive
the ratio is bounded away from $1$ and the theorem is proved.
\end{proof}

The second proof is more directly geometric, and gives an explicit upper bound for
$D$ in terms of geometric objects associated to the flow that live on $\CP^1$.

\begin{proof}
Denote the fundamental group of $M$ by $G$. Let $\tilde{X}$ be the lifted flow to
the universal cover $\tilde{M}$. The orbit space $\OO$ of any lifted pseudo-Anosov flow
$\tilde{X}$ is homeomorphic to a plane. Since a pseudo-Anosov flow without perfect fits
is quasigeodesic, there are endpoint maps $f^\pm:\OO \to Z^\pm \subset \CP^1$ and
since there are no perfect fits, the images $Z^\pm$ are disjoint (though they are not
closed).

The endpoint maps $f^\pm$ are the coordinates of a map from $\OO$ to $Z^+\times Z^-$.
The image $P$ is homeomorphic to a plane, and is properly embedded in $\CP^1\times \CP^1 -\Delta$
and therefore we obtain a $G$-invariant properly embedded closed subset 
$$\tilde{N}:= P \times \R \subset (\CP^1 \times \CP^1 - \Delta) \times \R = UT\HH^3$$
covering a closed subset $N$ of $UTM$ invariant under the geodesic flow.
Closed orbits of $X$ are in bijection with closed orbits of the geodesic flow on $N$;
in fact, $N$ foliated by the orbits of the geodesic flow is orbit equivalent to $M,X$.
Thus we are reduced to counting the number of closed orbits in $N$ of length $\le T$.

We must now use some results from dynamics, particularly 
Ledrappier--Young \cite{Ledrappier_Young_1,Ledrappier_Young_2}.
As is well-known, the exponential growth rate of the number of orbits of a closed
flow-invariant subset of $UTM$ is the topological entropy $D$ of this flow, which is
equal to the supremum of the (measure-theoretic) entropies of the flow-invariant
measures on $N$. The Ledrappier--Young formula is elementary in this case since there is
only one Lyapunov exponent, and it says that the measure-theoretic entropy of an 
invariant measure is  equal to the Hausdorff dimension of the conditional measure 
on the unstable manifold. The dimension of this conditional measure may be bounded
by pushing it forward to $\CP^1$ by taking the endpoint of the geodesic flow, and
estimating the box dimension of the measure on the image. 

Now, the box dimension of $Z^+$ is $2$ because $Z^+$ is dense. This may be rectified
by restricting attention to the intersection of $\tilde{N}$ with any compact subset
of $UT\HH^3$ and estimating the box dimension of the image under the endpoint map. 
This may be done by using the theory of CaTherine wheels; see \cite{Calegari_Loukidou_wheels}. 
Associated to $X$ there is a
CaTherine wheel $f:S^1 \to \CP^1$ invariant under the (geometric) action of $G$ on 
$\CP^1$. Since $G$ is a cocompact Kleinian group, there is a $K$ so that for every
interval $I\subset S^1$ the image $f(I)$ is a $K$-quasidisk. Every arc in $Z^\pm$ is
contained in $\partial f(I)$ for some $I$ as above, and each of $Z^\pm$ is a
countable union of such arcs. A $K$-quasiarc has equal box and 
Hausdorff dimensions, and either is at most $2K/(K+1)$ by Astala \cite{Astala}.

The image of a compact subset of $\tilde{N}$ under the endpoint map is a finite union
of compact arcs in $Z^+$. Each of these is a quasiarc, with box dimension at most
$2K/(K+1)$, and the box dimension of a finite union is the supremum of the box
dimensions of the components. The theorem follows.
\end{proof}

\section{Computing orbit counts}\label{section:computations}

What is interesting about Theorem~\ref{theorem:orbit_bound_hyp} 
is not so much that the exponential 
growth rate of closed orbits for $X$ is strictly less than the growth rate
of the ambient space, but that it is bounded by a meaningful geometric quantity,
namely the Hausdorff dimension of $Z^\pm$. Turning this around, if we compute (or estimate)
the growth rate of closed orbits in $X$, we get a lower bound on this
Hausdorff dimension. This may be compared with a numerical estimate of
the Hausdorff dimension of $Z^\pm$ to see how good the bound is in practice.

We tried to make this comparison effective for a particular family of examples, namely
the $(0,n)$-orbifold fillings on the figure eight knot complement. These give rise
to hyperbolic orbifolds with a single order $n$ orbifold geodesic; these orbifolds
fiber over $S^1$ with fiber a torus with a single orbifold point of order $n$, and
with monodromy $(\begin{smallmatrix} 2 & 1 \\ 1 & 1 \end{smallmatrix})$.

The associated pseudo-Anosov flows $X$ have no perfect fits; topologically, they are
`the same' flow (up to orbit equivalence) on the underlying topological manifold, though
the hyperbolic lengths of the conjugacy classes of the closed orbits depend in a
complicated way on $n$.

Here is how we obtain a numerical estimate of $\dim_H(Z^+)$. We may draw approximations
to a known arc $\alpha \subset Z^+ \subset \C$ recursively to greater and greater degree of 
accuracy by computing an $\epsilon$-mesh of points $z_{j,\epsilon}$ known to lie on $\alpha$
(in the correct order), computing $L(\epsilon):=\sum_j |z_{j+1,\epsilon} - z_{j,\epsilon}|$, 
estimating a linear regression of the form
$$\log L(\epsilon) = C_1 - C_2 \log(\epsilon)$$
and then using the heuristic $\dim_H(Z^+) = \dim_H(\alpha) = 1 + C_2$.

Here is how we obtain a lower bound on $D_X$, the growth rate of orbits in $X$. Choose
a Markov partition with associated directed graph $\Gamma$. We constructed these Markov partitions with the help of the \texttt{veering} package~\cite{Veering_Codebase}. Specify a
random walk on this directed graph by choosing weights on the edges of $\Gamma$. Let $H$ be
the entropy rate of this random walk. This random walk gives a distribution on closed 
paths of length $m$ on $\Gamma$, whose entropy is asymptotic to $mH$. Closed paths in $\Gamma$
correspond to closed orbits in $M$ up to a small $\text{poly}(m)$ 
overcounting due to cyclic shifts
of the path and overcounting at the boundary of the Markov partition. The geodesic lengths of the corresponding orbits concentrate near $mL$ for some constant $L$ which can
be estimated by sampling long random walks. Therefore, for $m$ large we have constructed a distribution of entropy $mH$ supported on orbits of length $(1+o(1))mL$. It follows that
the logarithm of the number of orbits of length $mL$ is at least $\sim mH$, and $H/L$ is
a lower bound for the exponential growth rate of orbits in the hyperbolic metric. We may
optimize our lower bound by refining the Markov partition and searching for different weights
on $\Gamma$. For the computations in Figure~\ref{fig:Hausdorff_dimension_bounds}, we used a Markov partition with 6 rectangles.

We have no theoretical guarantee that our estimate of Hausdorff dimension is even close
to accurate, and indeed our numerical estimates are in conflict for lower order 
orbifold fillings; see Figure~\ref{fig:Hausdorff_dimension_bounds}.

\begin{figure}
\begin{tabular}{c|cc}
$n$ & $\dim_H(Z^+)$  & $D_X$ (lower bound)\\
\hline
2  & 1.081 & 1.108\\
3  & 1.148 & 1.182\\
4  & 1.190 & 1.213\\
5  & 1.214 & 1.227\\
6  & 1.232 & 1.241\\
10 & 1.264 & 1.252\\
50 & 1.297 & 1.259\\
\end{tabular}
\caption{Comparing a lower bound for the growth rate of the flow $D_X$ with a
heuristic estimate of the Hausdorff dimension of $Z^+$
for orbifold fillings of the figure eight knot complement.}\label{fig:Hausdorff_dimension_bounds}
\end{figure}

Two qualitative features emerge from these estimates. Firstly, both quantities seem to
be monotone increasing with $n$. Secondly, both quantities seem to approach a limit as
$n \to \infty$; the convergence rate for $\dim_H(Z^+)$ appears to be of order $O(1/n)$, while the convergence for our lower bound on $D_X$ appears faster, perhaps of order $O(1/n^2)$, though we do not know how the tightness of our lower bound changes with $n$. We have no theoretical justification for these observations.

We also computed similar lower bounds for $D_X$ for a variety of pseudo-Anosov flows on small closed hyperbolic 3-manifolds from the Hodgson--Weeks census; see Figure~\ref{fig:growth_rate_bounds}.

\begin{figure}
\begin{tabular}{l|c}
flow  & $D_X$ (lower bound)\\
\hline
\codestyle|cPcbbbdxm_10(2, -3)| & 1.028\\
\codestyle|cPcbbbiht_12(8, 1)| & 1.571\\
\codestyle|dLQacccjsnk_200(3, -1)| & 1.065\\
\codestyle|dLQbccchhfo_122(4, -1)| & 1.554\\
\codestyle|dLQbccchhsj_122(4, -3)| & 1.070\\
\codestyle|eLAkaccddjsnak_2001(3, -1)| & 1.074\\
\codestyle|eLAkbbcdddhwqj_2102(2, 1)| & 1.054\\
\codestyle|eLAkbccddhhsqs_1220(1, -3)| & 1.326\\
\codestyle|eLMkbcddddedde_2100(3, -4)(3, -2)| & 1.203\\
\codestyle|eLMkbcdddhhhdu_1221(1, 1)|& 1.082\\
\codestyle|eLMkbcdddhhhml_1221(3, -1)|& 1.443\\
\codestyle|eLMkbcdddhhqqa_1220(3, -1)|& 1.074\\
\codestyle|eLMkbcdddhhqxh_1220(3, -1)|& 1.102\\
\codestyle|eLMkbcdddhxqdu_1200(1, -4)|& 1.536\\
\codestyle|eLMkbcdddhxqlm_1200(1, -3)|& 1.103\\
\codestyle|eLPkaccddjnkaj_2002(3, -4)|& 1.031\\
\codestyle|eLPkbcdddhrrcv_1200(3, -2)|& 1.048
\end{tabular}
\caption{Lower bound for $D_X$ for pseudo-Anosov flows on several small closed 3-manifolds. The flows are specified with the conventions of \cite{Veering_Census}.}\label{fig:growth_rate_bounds}
\end{figure}

\section*{Acknowledgements}

We would like to thank Thomas Barthelm\'e, Stephen Cantrell, 
Alex Eskin, Fran\c cois Ledrappier, Kathryn Mann, Rafael Potrie, 
Sam Taylor and Amie Wilkinson for helpful comments and feedback.


\begin{thebibliography}{99}
\bibitem{Astala}
K. Astala,
\emph{Area distortion of quasiconformal mappings},
Acta Math. {\bf 173} (1994), 37--60
\bibitem{BMPZ}
T. Barthelm\'e, K. Mann, N. Paulet and A. Zalloum,
\emph{Pseudo-Anosov flows and the geometry of Anosov-like group actions},
preprint, \url{https://arxiv.org/abs/2605.12837}
\bibitem{Brooks}
R. Brooks,
\emph{Some remarks on bounded cohomology},
Ann. Math. Stud. {\bf 97}, Princeton University Press, Princeton, 1981. 53--63
\bibitem{Calegari_scl}
D. Calegari,
\emph{scl},
MSJ Memoirs {\bf 20} Mathematical Society of Japan, Tokyo, 2009.
\bibitem{Calegari_Fujiwara}
D. Calegari and K. Fujiwara,
\emph{Combable functions, quasimorphisms, and the central limit theorem},
Erg. Thy Dyn. Sys. {\bf 30} (2010), no. 5, 1343--1369
\bibitem{Calegari_ergodic}
D. Calegari,
\emph{The ergodic theory of hyperbolic groups},
Contemp. Math., {\bf 597}
AMS, Providence, RI, 2013. 15--52
\bibitem{Calegari_Loukidou_zippers}
D. Calegari and I. Loukidou,
\emph{Zippers},
Geom. Top. to appear; \url{https://arxiv.org/abs/2411.15610}
\bibitem{Calegari_Loukidou_wheels}
D. Calegari and I. Loukidou,
\emph{CaTherine wheels},
preprint, \url{https://arxiv.org/abs/2604.24619}
\bibitem{Cantrell_Reyes_Sert}
S. Cantrell, E. Reyes and C. Sert,
\emph{Growth indicators, Manhattan manifolds, and translation cones for hyperbolic groups},
preprint, to appear
\bibitem{Fenley_quasigeodesic}
S. Fenley,
\emph{Quasigeodesic pseudo-Anosov flows in hyperbolic 3-manifolds and connections with large scale geometry},
Topology {\bf 40} (2001) no. 3, 503--537
\bibitem{Fujiwara}
K. Fujiwara,
\emph{The second bounded cohomology of a group acting on a Gromov-hyperbolic space},
Proc. LMS (3) {\bf 76} (1998), no. 1, 70--94
\bibitem{Gromov}
M. Gromov,
\emph{Hyperbolic groups},
MSRI Pub. {\bf 8} Springer-Verlag, New York, 1987. 75--263
\bibitem{Jezequel_Zung}
M. J\'ez\'equel and J. Zung,
\emph{Zeta functions and the Fried conjecture for smooth pseudo-Anosov flows},
preprint, \url{https://arxiv.org/abs/2409.17014}
\bibitem{Katok}
A. Katok and B. Hasselblatt, 
\emph{Introduction to the Modern Theory of Dynamical Systems},
Cambridge University Press, Cambridge, 1995.
\bibitem{Kemeny_Snell}
J. Kemeny and L. Snell, 
\emph{Finite Markov Chains}, 
University series in Undergraduate Mathematics, Van Nostrand, 1960.
\bibitem{Landry_Minsky_Taylor}
M. Landry, Y. Minsky and S. Taylor,
\emph{Flows, growth rates, and the veering polynomial},
Erg. Thy Dyn. Sys. {\bf 43} (2023) no. 9, 3026--3107
\bibitem{Veering_Census}
A. Giannopoulos, S. Schleimer, and H. Segerman,
\emph{A census of veering structures}
\url{https://math.okstate.edu/people/segerman/veering.html}, 2019.
\bibitem{Ledrappier_Young_1}
F. Ledrappier and L.-S. Young,
\emph{The metric entropy of diffeomorphisms I Characterization of measures 
satisfying Pesin's entropy formula},
Ann. Math. (2) {\bf 122} (1985), no. 3, 509--539
\bibitem{Ledrappier_Young_2}
F. Ledrappier and L.-S. Young,
\emph{The metric entropy of diffeomorphisms II Relations between entropy, 
exponents and dimension},
Ann. Math. (2) {\bf 122} (1985), no. 3, 540--574
\bibitem{Mineyev}
I. Mineyev,
\emph{Flows and joins of metric spaces},
Geom. Top. {\bf 9} (2005), 403--482
\bibitem{Veering_Codebase}
A. Parlak, S. Schleimer, and H. Segerman
\emph{Veering 0.5, code for studying taut and veering ideal triangulations},
\url{https://github.com/henryseg/Veering}, 2026.
\bibitem{Parry_Pollicott}
W. Parry and M. Pollicott,
\emph{Zeta functions and the periodic orbit structure of hyperbolic dynamics},
Asterisque, {\bf 187--88} (1990), 1--268
\bibitem{Pollicott_Sharp}
M. Pollicott and R. Sharp, 
\emph{Comparison theorems and orbit counting in hyperbolic geometry},
Trans. AMS {\bf 350} (1998), no. 2, 473--499
\bibitem{Rhemtulla}
A. Rhemtulla,
\emph{A problem of bounded expressibility in free products},
Proc. Camb. Phil. Soc. {\bf 64} (1968), 573--584
\bibitem{Ruelle}
D. Ruelle,
\emph{Thermodynamic Formalism},
Addison Wesley, New York, 1978.
\end{thebibliography}
\end{document}